\documentclass[letterpaper,11pt]{article}

\usepackage[runin]{abstract}
\usepackage{geometry}
\usepackage{titlesec}
\usepackage{graphicx}
\usepackage{amsmath}
\usepackage{amsthm}
\usepackage{amsfonts}
\usepackage{amssymb}
\usepackage{empheq}
\usepackage{algorithmic}
\usepackage{algorithm}
\usepackage[authoryear]{natbib}
\usepackage{url}

\usepackage[font=sf]{caption}

\titleformat*{\section}{\Large\bfseries}
\titleformat*{\subsection}{\large\sc}
\titleformat*{\subsubsection}{\itshape}

\usepackage{mathtools}

\usepackage{epigraph}

\begin{document}

\title{{\bf On the tractability of Nash equilibrium}}

\author{\large{Ioannis Avramopoulos}}

\date{}

\maketitle

\thispagestyle{empty} 

\newtheorem{definition}{Definition}
\newtheorem{proposition}{Proposition}
\newtheorem{theorem}{Theorem}
\newtheorem*{theorem*}{Theorem}
\newtheorem{corollary}{Corollary}
\newtheorem{lemma}{Lemma}
\newtheorem{axiom}{Axiom}
\newtheorem{thesis}{Thesis}

\vspace*{-0.2truecm}

\begin{abstract}
In this paper, we propose a method for solving a PPAD-complete problem \citep{PPAD}. Given is the payoff matrix $C$ of a symmetric bimatrix game $(C, C^T)$ and our goal is to compute a Nash equilibrium of $(C, C^T)$. In this paper, we devise a nonlinear replicator dynamic (whose right-hand-side can be obtained by solving a pair of convex optimization problems) with the following property: Under any invertible $0 < C \leq 1$, every orbit of our dynamic starting at an interior strategy of the standard simplex approaches a set of strategies of $(C, C^T)$ such that, for each strategy in this set, a symmetric Nash equilibrium strategy can be computed by solving the aforementioned convex mathematical programs. We prove convergence using results in analysis (the analytic implicit function theorem), nonlinear optimization theory (duality theory, Berge's maximum principle, and a theorem of \cite{Robinson} on the Lipschitz continuity of parametric nonlinear programs), and dynamical systems theory (such as the LaSalle invariance principle).
\end{abstract}


\section{Introduction}

Game theory is a mathematical discipline concerned with the study of algebraic, analytic, and other objects that abstract the physical world, especially social interactions. The most important solution concept in game theory is the Nash equilibrium \citep{Nash}, a strategy profile (combination of strategies) in an $N$-player game such that no unilateral player deviations are profitable. The Nash equilibrium is an attractive solution concept, for example, as Nash showed, an equilibrium is guaranteed to exist in any $N$-person game. Over time this concept has formed a basic cornerstone of {\em economic theory,} but its reach extends beyond economics to the natural sciences and biology. 

In this paper, we propose a method, as an {\em evolutionary dynamic} \citep{PopulationGames}, to compute a symmetric Nash equilibrium in a symmetric bimatrix game. This problem is PPAD-complete as follows from \citep{CDT, Daskalakis} and a standard (folklore) symmetrization result of bimatrix games. We assume the payoff matrix to be invertible but this is only a mild assumption as singular matrices can be approximated arbitrarily well by invertible matrices. We do not attempt to ground the intuition of our equilibrium computation method from the perspective of economic theory. For example, proposing a microfoundation for our dynamic is beyond the scope of this paper. Methods for discretizing our equation are also out scope of our paper. We note that standard methods (such as the Euler or higher order methods such as Runge Kutta) certainly apply.

\subsection{Preliminaries in game theory}

Given a symmetric bimatrix game $(C, C^T)$, we denote the corresponding standard (probability) simplex by $\mathbb{X}(C)$. $\mathbb{\mathring{X}}(C)$ denotes the relative interior of $\mathbb{X}(C)$. The elements (probability vectors) of $\mathbb{X}(C)$ are called {\em strategies}. We call the standard basis vectors in $\mathbb{R}^n$ {\em pure strategies} and denote them by $E_i, i =1, \ldots, n$. {\em Symmetric Nash equilibria} are precisely those combinations of strategies $(X^*, X^*)$ such that $X^*$ satisfies $\forall Y \in \mathbb{X}(C) : X^* \cdot CX^* - Y \cdot CX^* \geq 0$. We call $X^*$ a {\em symmetric Nash equilibrium strategy}. A symmetric Nash equilibrium is guaranteed to always exist \citep{Nash2}. A simple fact is that $X^*$ is a symmetric Nash equilibrium strategy if and only if $(CX^*)_{\max} - X^* \cdot CX^* = 0$. If $(CX^*)_{\max} - X^* \cdot CX^* \leq \epsilon$, $X^*$ is called an $\epsilon$-approximate equilibrium strategy. A symmetric bimatrix game $(C, C^T)$ is zero-sum if $C$ is antisymmetric, that is, $C = -C^T$, and it is called doubly symmetric if $C$ is symmetric, that is, $C = C^T$. In this paper, we are concerned with a general $C$.

\begin{definition}
\label{fixed_point}
$X$ is a fixed point of $(C, C^T)$, if, for $i \in \{1, \ldots, n\}$, $X(i) > 0 \Rightarrow (CX)_i = X \cdot CX$. 
\end{definition}

\noindent
Note that pure strategies and symmetric Nash equilibria are necessarily fixed points of $(C, C^T)$.

\subsection{Our techniques and main result}

A plausible approach to compute a Nash equilibrium in this setting is to use the replicator dynamic
\begin{align*}
\dot{X}(i) = X(i) ((CX)_i - X \cdot CX), \quad i=1,\ldots,n.
\end{align*}
In this paper, we consider the more general dynamic
\begin{align*}
\dot{X}(i) = X(i) ((CZ)_i - X \cdot CZ), \quad i=1,\ldots,n,
\end{align*}
where $Z \in \mathbb{X}(C)$. We call parameter $Z$ the {\em multiplier} at $X$. Our main result has as follows:

\begin{theorem}
\label{main_theorem}
Let $0 < C \leq 1$ be the payoff matrix of a symmetric bimatrix game $(C, C^T)$ and let $S \equiv CC^T$. Furthermore, let $RE(\cdot, \cdot)$ denote the relative entropy function. If $C$ is invertible, starting at any $X^0 \in \mathbb{\mathring{X}}(C)$, every limit point of the orbit $X_t, t \in [0, +\infty)$ of
\begin{align}
\dot{X}(i) = X(i) \left( (CZ_X)_i - X \cdot CZ_X \right)\label{principal_equation}
\end{align}
as $t \rightarrow \infty$, is a strategy, say $X$, whose multiplier $Z_X$ is a symmetric Nash equilibrium strategy of $(C, C^T)$. For any $X$ in the orbit, $Z_X \in \mathbb{X}(C)$ is the unique minimizer of the convex optimization problem
\begin{align}
\mbox{minimize} \quad &RE(Z, X) + \frac{1}{2} \epsilon^2 + \frac{1}{2} \eta^2\label{first_expression}\\
\mbox{subject to} \quad &(X \cdot SX)^{-1} \sum_{i=1}^n X(i) (SX)_i (CZ)_i - X \cdot CZ \geq \epsilon \theta^2_X\label{first_constraint}\\
&(CZ)_{\max} - X \cdot CZ \leq \epsilon\label{second_constraint}\\
&\ln \left( \sum_{i=1}^n X(i) \exp\{\alpha (CZ)_i\} \right) - \alpha X \cdot CZ \leq \eta^{\theta_X} - 1\label{third_constraint}\\
&\ln \left( \sum_{i=1}^n X(i) \exp\{\alpha (SZ)_i\} \right) - \alpha X \cdot SZ \leq \eta^{\theta_X} - 1\label{fourth_constraint}\\
&\eta \geq 1, \quad Z \in \mathbb{X}(C),\notag
\end{align}
where $\alpha > 0$ is a constant and $\theta_X \geq 0$ is the unique optimizer of the also convex optimization problem
\begin{align}
\mbox{minimize} \quad &RE(Z, X) - \theta\label{optproblem}\\
\mbox{subject to} \quad &(X \cdot SX)^{-1} \sum_{i=1}^n X(i) (SX)_i (CZ)_i - X \cdot CZ \geq \theta^2\label{optproblemconstraint}\\
\quad &Z \in \mathbb{X}(C).\notag
\end{align}
\end{theorem}

\subsection{Other related work}

Our dynamic is a replicator dynamic, and, therefore, it is related to the multiplicative weights update method \citep{AHK}. The replicator dynamic and the multiplicative weights method have various applications in theoretical computer science, artificial intelligence, and game theory. On the latter front, there is significant amount of work on Nash equilibrium computation and especially in the setting of $2$-player games. We simply mention a boundary of those results. The Lemke-Howson algorithm for computing an equilibrium in a bimatrix game is considered by many to be the state of the art in exact equilibrium computation but it has been shown to run in exponential time in the worst case \citep{Savani-Stengel}. There is a quasi-polynomial algorithm for additively approximate Nash equilibria in bimatrix games due to \cite{LMM} (based on looking for equilibria of small support on a grid). \cite{AvramopoulosBoosting} give a constructive existence proof of Nash equilibrium using the technique of multipliers (as we do in this paper). We note that the idea of using multipliers, in particular, in the related to this paper exponential multiplicative weights method, has independently been explored in \citep{Clairvoyant, AvramopoulosSSRN}. We also note that our method circumvents a tough impossibility result in game dynamics of \cite{Milionis}. Insofar, the best polynomial-time approximation algorithm for a Nash equilibrium achieves a 0.3393 approximation \citep{Tsaknakis-Spirakis-journal}.\\

\noindent
Sections \ref{preliminaries2} and \ref{preliminaries3} set the stage and Section \ref{proof_of_theorem} proves our main result (Theorem \ref{main_theorem}).

\section{On optimization problem \eqref{optproblem}}
\label{preliminaries2}

Let us first briefly mention some properties of the relative entropy function. The relative entropy between the $n \times 1$ probability vectors $P > 0$ (that is, for all $i = 1, \ldots, n$, $P(i) > 0$) and $Q > 0$ is given by 
\begin{align*}
RE(P, Q) \equiv \sum_{i=1}^n P(i) \ln \frac{P(i)}{Q(i)}.
\end{align*}
However, this definition can be relaxed: Let $\mathcal{C}(\cdot)$ denote the {\em carrier} of its argument (that is always a probability vector), where the carrier is the set of indices with positive probability mass. The relative entropy between $n \times 1$ probability vectors $P$ and $Q$ where $Q \in \{ \mathcal{Q} \in \mathbb{X} | \mathcal{C}(P) \subset \mathcal{C}(\mathcal{Q}) \}$ and $\mathbb{X}$ is a probability simplex of appropriate dimension, is
\begin{align*}
RE(P, Q) \equiv \sum_{i \in \mathcal{C}(P)} P(i) \ln \frac{P(i)}{Q(i)}.
\end{align*}
We note the well-known properties of the relative entropy \cite[p.96]{Weibull} that {\em (i)} $RE(P, Q) \geq 0$, {\em (ii)} $RE(P, Q) \geq \| P - Q \|^2$, where $\| \cdot \|$ is the Euclidean distance, {\em (iii)} $RE(P, P) = 0$, and {\em (iv)} $RE(P, Q) = 0$ iff $P = Q$. Note {\em (i)} follows from {\em (ii)} and {\em (iv)} follows from {\em (ii)} and {\em (iii)}. Note finally that
\begin{align*}
\mathcal{C}(P) \not\subset \mathcal{C}(\mathcal{Q}) \Rightarrow RE(P, Q) = \infty.
\end{align*}

\begin{lemma}
\label{fundamental_lemma}
The following are properties of optimization problem \eqref{optproblem}:\\

\noindent
(i) $\forall X \in \mathbb{\mathring{X}}(C)$, \eqref{optproblem} is feasible.\\

\noindent
(ii) On the boundary of $\mathbb{X}(C)$, \eqref{optproblem} is not guaranteed to be feasible.\\

\noindent
(iii) Wherever \eqref{optproblem} is feasible, $\theta_X \geq 0$ and $\theta_X$ is unique.\\

\noindent
(iv) $\theta_X > 0$ implies that $\theta_X$ is locally an analytic function of $X$.\\

\noindent
(v) Wherever \eqref{optproblem} is feasible, the objective is finite.
\end{lemma}

\begin{proof}
Our proof of this lemma rests on the dual formulation of \eqref{optproblem}. Taking the dual of \eqref{optproblem}, the Lagrangian is
\begin{align}
\mathcal{L}(Z, \theta, \alpha, \lambda) = RE(Z, X) - \theta - \alpha \left( (X \cdot SX)^{-1} \sum_{i=1}^n X(i) (SX)_i (CZ)_i - X \cdot CZ - \theta^2 \right) + \lambda (\mathbf{1}^T Z - 1),\label{langrangian}
\end{align}
where we assume that the constraints $Z \geq 0$ and $\theta \geq 0$ are implicit. The dual function $\mathcal{L}_D(\alpha, \lambda)$ is obtained by minimizing the Lagrangian $\mathcal{L}(Z, \theta, \alpha, \lambda)$, that is,
\begin{align*}
\mathcal{L}_D(\alpha, \lambda) = \inf \{ \mathcal{L}(Z, \theta, \alpha, \lambda) | Z \in \mathbb{R}^n, \theta \in \mathbb{R} \}.
\end{align*}
The Lagrangian is minimized when the gradient is zero. Let us first set the partial derivative with respect to $\theta$ equal to zero to obtain
\begin{align}
1 - 2 \alpha \theta = 0 \Rightarrow \theta = \frac{1}{2\alpha},\label{some}
\end{align}
which indeed satisfies $\theta \geq 0$ since $\alpha \geq 0$ (by the definition of the dual function), but \eqref{some} further implies that $\alpha > 0$ and, therefore, $\theta > 0$ unless $\alpha = \infty$. Let us now consider the partial derivative with respect to $Z$. Observing to that end that
\begin{align*}
\frac{\partial RE(Z, X)}{\partial Z(i)} = \frac{\partial }{\partial Z(i)} \left( Z(i) \ln \frac{Z(i)}{X(i)} \right) = 1 + \ln \frac{Z(i)}{X(i)}
\end{align*}
we obtain that the gradient is zero when
\begin{align*}
1 + \ln \frac{Z(i)}{X(i)} - \alpha \left( (X \cdot SX)^{-1} \sum_{j=1}^n C^T_{ij} X(j) (SX)_j - (C^TX)_i \right) + \lambda = 0.
\end{align*}
Solving for $Z(i)$ in the previous expression, we obtain
\begin{align}
Z(i) = X(i) \exp\left\{ - 1 - \lambda + \alpha \left( (X \cdot SX)^{-1} \sum_{j=1}^n C^T_{ij} X(j) (SX)_j - (C^TX)_i \right)  \right\} \quad i = 1, \ldots, n.\label{firm}
\end{align}
Note that, by the previous expression, the constraint $Z \geq 0$ is automatically satisfied. Using the notation
\begin{align*}
\mathcal{B}_i \equiv (X \cdot SX)^{-1} \sum_{j=1}^n C^T_{ij} X(j) (SX)_j - (C^TX)_i, \quad i =1, \ldots, n,
\end{align*}
noting that
\begin{align*}
(X \cdot SX)^{-1} \sum_{i=1}^n X(i) (SX)_i (CZ)_i - X \cdot CZ = \sum_{i=1}^n Z(i) \mathcal{B}_i,
\end{align*}
and substituting now the previous expressions for $Z$ and $\theta$ in \eqref{langrangian}, we obtain the dual function
\begin{align*}
\mathcal{L}_D(\alpha, \lambda) = &\sum_{i=1}^n X(i) \exp\{ - 1 - \lambda + \alpha \mathcal{B}_i \} (- 1 - \lambda + \alpha \mathcal{B}_i) -\\
  &- \sum_{i=1}^n X(i) \exp\{ - 1 - \lambda + \alpha \mathcal{B}_i  \} \alpha \mathcal{B}_i + \\
  &+ \lambda \left(\sum_{i=1}^n X(i) \exp\{ - 1 - \lambda + \alpha \mathcal{B}_i \} - 1\right) +\\
  &- \frac{1}{2\alpha} + \alpha \left( \frac{1}{2\alpha} \right)^2
\end{align*}
which simplifies to
\begin{align*}
\mathcal{L}_D(\alpha, \lambda) = &- \sum_{i=1}^n X(i) \exp\{ - 1 - \lambda + \alpha \mathcal{B}_i \} -\lambda - \frac{1}{4 \alpha}.
\end{align*}
This function is concave in $\alpha$ and $\lambda$. Since the dual function is concave, to find the optimal $\alpha$ and $\lambda$ we simply need to set the derivatives of $\mathcal{L}_D(\alpha, \lambda)$ (with respect to $\alpha$ and $\lambda$) equal to $0$ (by the previous observation that $\alpha$ is necessarily strictly positive). To that end, we have
\begin{align*}
\frac{d \mathcal{L}_D(\alpha, \lambda)}{d \lambda} = \sum_{i=1}^n X(i) \exp\{ - 1 - \lambda + \alpha\mathcal{B} _i \} - 1 = 0,
\end{align*}
which implies
\begin{align*}
\exp\{ - 1 - \lambda \} = \frac{1}{\sum_{i=1}^n X(i) \exp\{ \alpha \mathcal{B}_i \}}.
\end{align*}
Substituting in \eqref{firm} we obtain (by the zero duality gap since the problem is convex) the optimal $Z_X$ as
\begin{align}
Z_X(i) = X(i) \frac{ \exp\{\alpha \mathcal{B}_i\}}{\sum_{i=1}^n X(i) \exp\{ \alpha \mathcal{B}_i \}} \quad i=1, \ldots, n.\label{optimalZ}
\end{align}
Taking now the derivative of $\mathcal{L}_D(\alpha, \lambda)$ with respect to $\alpha$ and setting it equal to zero, we obtain that
\begin{align}
- \sum_{i=1}^n X(i) \mathcal{B}_i \frac{\exp\{ \alpha \mathcal{B}_i \}}{\sum_{i=1}^n X(i) \exp\{ \alpha \mathcal{B}_i \}} + \frac{1}{4 \alpha^2} = 0\label{fixpoint}
\end{align}
\eqref{fixpoint} defines an implicit function. Note now that, by straight calculus and Jensen's inequality, we have
\begin{align*}
\frac{d}{d\alpha} \left(\sum_{i=1}^n X(i) \mathcal{B}_i \frac{\exp\{ \alpha \mathcal{B}_i \}}{\sum_{i=1}^n X(i) \exp\{ \alpha \mathcal{B}_i \}} \right) \geq 0.
\end{align*}
Let us show these steps: We have
\begin{align*}
\frac{d}{d\alpha} \left(\frac{\sum_{i=1}^n X(i) \mathcal{B}_i \exp\{ \alpha \mathcal{B}_i \}}{\sum_{i=1}^n X(i) \exp\{ \alpha \mathcal{B}_i \}} \right) &= \frac{\left( \sum_{i=1}^n X(i) \mathcal{B}^2_i \exp\{ \alpha \mathcal{B}_i \} \right) \left( \sum_{i=1}^n X(i) \exp\{ \alpha \mathcal{B}_i \} \right)}{\left( \sum_{i=1}^n X(i) \exp\{ \alpha \mathcal{B}_i \} \right)^2} -\\
&- \frac{\left( \sum_{i=1}^n X(i) \mathcal{B}_i \exp\{ \alpha \mathcal{B}_i \} \right)^2}{\left( \sum_{i=1}^n X(i) \exp\{ \alpha \mathcal{B}_i \} \right)^2}
\end{align*}
which is positive since, by Jensen's inequality, we have
\begin{align*}
\frac{\sum_{i=1}^n X(i) \mathcal{B}^2_i \exp\{ \alpha \mathcal{B}_i \}}{\sum_{i=1}^n X(i) \exp\{ \alpha \mathcal{B}_i \}} \geq \left( \frac{\sum_{i=1}^n X(i) \mathcal{B}_i \exp\{ \alpha \mathcal{B}_i \}}{\sum_{i=1}^n X(i) \exp\{ \alpha \mathcal{B}_i \}} \right)^2
\end{align*}
which implies that
\begin{align*}
\left( \sum_{i=1}^n X(i) \mathcal{B}^2_i \exp\{ \alpha \mathcal{B}_i \} \right) \left( \sum_{i=1}^n X(i) \exp\{ \alpha \mathcal{B}_i \} \right) \geq \left( \sum_{i=1}^n X(i) \mathcal{B}_i \exp\{ \alpha \mathcal{B}_i \} \right)^2.
\end{align*}
Therefore, the derivative (Jacobian) of \eqref{fixpoint} with respect to $\alpha$ is always $< 0$ provided $\alpha$ is finite.\\

\noindent
Let us now show that
\begin{align}
\lim_{\alpha \rightarrow \infty} \left\{ \frac{\sum_{i=1}^n X(i) \mathcal{B}_i \exp\{ \alpha \mathcal{B}_i \}}{\sum_{i=1}^n X(i) \exp\{ \alpha \mathcal{B}_i \}} \right\} = \max_{Z \in \mathbb{X}(C_X)} \left\{ (X \cdot SX)^{-1} \sum_{i=1}^n X(i) (SX)_i (CZ)_i - X \cdot CZ \right\},\label{hyper}
\end{align}
where $\mathbb{X}(C_X)$ is the carrier of $X$. To that end, note first that, by \eqref{optimalZ}, we have
\begin{align*}
\frac{\sum_{i=1}^n X(i) \mathcal{B}_i \exp\{ \alpha \mathcal{B}_i \}}{\sum_{i=1}^n X(i) \exp\{ \alpha \mathcal{B}_i \}} = \sum_{i=1}^n Z_X(i) \mathcal{B}_i.
\end{align*}
Furthermore, we may write \eqref{optimalZ} as
\begin{align*}
Z_X(i) = X(i) \frac{1}{\sum_{j=1}^n X(j) \exp\{ \alpha (\mathcal{B}_j - \mathcal{B}_i) \}} \quad i=1, \ldots, n
\end{align*}
Taking the limit as $\alpha \rightarrow \infty$, we obtain that, in the carrier of $X$,
\begin{align*}
\mathcal{B}_i < \mathcal{B}_{\max} \Rightarrow Z_X(i) = 0
\end{align*}
which implies that, in the carrier of $X$,
\begin{align*}
Z_X(i) > 0 \Rightarrow \mathcal{B}_i = \mathcal{B}_{\max}.
\end{align*}
Therefore, by straight algebra using the definition of $\mathcal{B}_i, i =1,\ldots, n$, the optimal $Z_X$ is obtained at $\alpha = \infty$ only if 
\begin{align}
Z_X \in &\arg\max\left\{ \sum_{i=1}^n Z(i) \mathcal{B}_i \bigg| Z \in \mathbb{X}(C_X) \right\} =\notag\\
&= \arg\max \left\{ (X \cdot SX)^{-1} \sum_{i=1}^n X(i) (SX)_i (CZ)_i - X \cdot CZ \bigg| Z \in \mathbb{X}(C_X) \right\}\label{endum}.
\end{align}
This proves \eqref{hyper}.\\

\noindent
Note, further, that
\begin{align}
\max \left\{ (X \cdot SX)^{-1} \sum_{i=1}^n X(i) (SX)_i (CZ)_i - X \cdot CZ \bigg| Z \in \mathbb{X}(C) \right\} \geq 0\label{max_inequality}
\end{align}
because, letting
\begin{align*}
Z = \left( \sum_{i=1}^n (C^TX)_i \right) C^TX,
\end{align*}
we obtain, by Jensen's inequality, that the expression
\begin{align*}
(X \cdot SX)^{-1} \sum_{i=1}^n X(i) (SX)_i (SX)_i - X \cdot SX
\end{align*}
is strictly positive unless $X$ is a fixed point of $(S, S)$.\\

\noindent
We are now ready to prove the lemma. Let us first assume that $X \in \mathbb{\mathring{X}}(C)$. Observe that when $\alpha = 0$, the left-hand-side of \eqref{fixpoint} is $+\infty$, and, when $\alpha = +\infty$, the left-hand-side of \eqref{fixpoint} is $\leq 0$ (as follows from \eqref{hyper} and \eqref{max_inequality}). Since the the left-hand-side of \eqref{fixpoint} is a strictly monotonically decreasing function of $\alpha$, we obtain that, given any $X \in \mathbb{X}(C)$, there exists a (unique) $\alpha$ that satisfies \eqref{fixpoint}. This proves (i). (ii) is implied by \eqref{hyper} since we cannot eliminate the possibility that there may exist an $X$ on the boundary of $\mathbb{X}(C)$ such that
\begin{align*}
\max_{Z \in \mathbb{X}(C_X)} \left\{ (X \cdot SX)^{-1} \sum_{i=1}^n X(i) (SX)_i (CZ)_i - X \cdot CZ \right\} < 0
\end{align*}
in which case there is no $\alpha > 0$ that satisfies \eqref{fixpoint}. (iii) is implied by \eqref{some}, that $\alpha \geq 0$ by the definition of the dual function, and by the previous observation that the derivative (Jacobian) of \eqref{fixpoint} with respect to $\alpha$ is always $< 0$ provided $\alpha$ is finite. Finally, (iv) is implied by the same observation that the derivative (Jacobian) of \eqref{fixpoint} with respect to $\alpha$ is always $< 0$ provided $\alpha$ is finite and the $\mathcal{C}^{\infty}$ (analytic) implicit function theorem. Finally, to prove (v), note that the optimal $Z_X$ is on the carrier of $X$ if $\alpha > 0$ is finite (see \eqref{optimalZ}). That $Z_X$ is on the carrier of $X$ if $\alpha = \infty$ is implied by \eqref{endum}. Therefore, (v) holds necessarily. This completes the proof of the lemma.
\end{proof}

\section{On optimization problem \eqref{first_expression}}
\label{preliminaries3}

\begin{lemma}
\label{convexxxity_lemma}
$\forall X \in \mathbb{X}(C)$ and $\forall \alpha > 0$, we have that
\begin{align*}
\ln \left( \sum_{i=1}^n X(i) \exp\{\alpha (CZ)_i \} \right)
\end{align*}
is a convex function of $Z$.
\end{lemma}

\begin{proof}
It is well known that the logsumexp function
\begin{align*}
\ln \left( \sum_{i=1}^n \exp\{Z(i)\} \right)
\end{align*}
is convex. It follows easily by a straightforward calculation of the Hessian that the function
\begin{align*}
f(Z) = \ln \left( \sum_{i=1}^n X(i) \exp\{Z(i)\} \right)
\end{align*}
is also convex in $Z$. The lemma then follows by noting that $f(\alpha CZ)$ is also convex.
\end{proof}

\begin{lemma}
\label{equality_lemma_1}
Regarding \eqref{first_expression}, we have that
\begin{align}
\ln \left( \sum_{i=1}^n X(i) \exp\{\alpha (CZ)_i\} \right) - \alpha X \cdot CZ = 0\label{ena000}
\end{align}
if and only if
\begin{align*}
\forall i \in \{1, \ldots, n\} : \left\{ X(i) > 0 \Rightarrow (CZ)_i = X \cdot CZ \right\}.
\end{align*}
Therefore, if $Z$ is a pure strategy in the carrier of $X$, \eqref{ena000} is necessarily true.
\end{lemma}

\begin{proof}
The lemma follows from the case of equality in Jensen's inequality and straight algebra.
\end{proof}

\begin{lemma}
\label{equality_lemma_2}
Regarding \eqref{first_expression}, we have that
\begin{align}
\ln \left( \sum_{i=1}^n X(i) \exp\{\alpha (SZ)_i\} \right) - \alpha X \cdot SZ = 0\label{ena001}
\end{align}
if and only if
\begin{align*}
\forall i \in \{1, \ldots, n\} : \left\{ X(i) > 0 \Rightarrow (SZ)_i = X \cdot SZ \right\} 
\end{align*}
Therefore, if $Z$ is a pure strategy in the carrier of $X$, \eqref{ena001} is necessarily true.
\end{lemma}

\begin{proof}
Identical to the proof of Lemma \ref{equality_lemma_1}.
\end{proof}

\begin{lemma}
\label{feasibility_lemma}
For all $X \in \mathbb{X}(C)$ such that $\theta_X \geq 0$, \eqref{first_expression} is feasible and $RE(Z_X, X)$ is finite.
\end{lemma}

\begin{proof}
Let $(Y_X, \theta_X)$ denote the solution of optimization problem \eqref{optproblem}. Let us first assume $\theta_X > 0$. Then, $RE(Y_X, X)$ is finite by Lemma \ref{fundamental_lemma} (v). Furthermore, optimization problem \eqref{first_expression} is feasible at $Z = Y_X$, $\epsilon = 1$ (since $0 < C \leq 1$), and some $\eta > 0$ since the function $\eta^{\theta_X}-1$ is monotonically increasing and diverges as $\eta \rightarrow \infty$. This proves the lemma assuming $\theta_X > 0$. Let us now assume that $\theta_X = 0$. Then, optimization problem \eqref{first_expression} is feasible if $Z$ is any pure strategy in the carrier of $X$, $\epsilon = 1$, and $\eta = 1$, since constraints \eqref{first_constraint}, \eqref{third_constraint}, and \eqref{fourth_constraint} are then (weakly) feasible (by Lemmas \ref{equality_lemma_1} and \ref{equality_lemma_2}). But for those values of our optimization variables, $RE(Z, X)$ is finite. Therefore, at the optimal $(Z_X, \epsilon_X, \eta_X)$,  $RE(Z_X, X)$ will remain finite. This completes the proof of our lemma.
\end{proof}

\noindent
Following \cite{Giorgi}, we consider a nonlinear program of the form
\begin{align*}
\mbox{ minimize } &\quad f(Z)\\
\mbox{ subject to } &\quad g_i(Z) \leq 0, i = 1, \ldots, m\\
&\quad h_j(Z) = 0, j =1, \ldots, p,
\end{align*}
where $f$, $g_i, i =1, \ldots, m$, and $h_j, j=1, \ldots, p$ are $\mathcal{C}^2$ (twice continuously differentiable). The Lagrangian associated with the previous problem is
\begin{align*}
\mathcal{L}(Z, u, w) = f(Z) + \sum_{i=1}^m u_i g_i(Z) + \sum_{j=1}^p w_j h_j(Z)
\end{align*}
where $u = [u_1, \ldots, u_m]^T$ and $w = [w_1, \ldots, w_p]^T$ are the multiplier vectors.

\begin{definition}
We say that the strong second order sufficient conditions hold at a feasible $Z$ with multipliers $(u, w)$ if
\begin{align*}
Y \cdot \nabla_{ZZ}^2 \mathcal{L}(Z, u, w)Y > 0
\end{align*}
for all $Y \neq 0$ such that
\begin{align*}
&\nabla g_i(Z) Y = 0, \quad \mbox{ if } u_i > 0\\
&\nabla h_j(Z) Y = 0, \quad j = 1, \ldots, p.
\end{align*}
\end{definition}

Let us now consider the parametric nonlinear program
\begin{align*}
\mbox{ minimize } &\quad f(Z, X)\\
\mbox{ subject to } &\quad g_i(Z, X) \leq 0, i = 1, \ldots, m\\
&\quad h_j(Z, X) = 0, j =1, \ldots, p,
\end{align*}
where $f$, $g_i, i =1, \ldots, m$, and $h_j, j=1, \ldots, p$ are $\mathcal{C}^2$ in both $Z$ and $X$. \cite{Robinson} (see also \citep[Theorem 8]{Giorgi}) shows that if $Z_X$, as a critical point of the aforementioned nonlinear program, satisfies the second order sufficient conditions, then $Z_X$ is a locally Lipschitz function of $X$.

\begin{lemma}
\label{lipschitz_lemma0}
$\theta_X > 0$ implies that the solution $Z_X$ of optimization problem \eqref{first_expression} is a locally Lipschitz function of $X$.
\end{lemma}

\begin{proof}
Lemma \ref{fundamental_lemma} (iv) gives that $\theta_X > 0$ is an analytic function of $X$. This implies that \eqref{first_expression} then satisfies the requirement that the constraints are twice continuously differentiable. It only then remains to show that the second order sufficient condition is satisfied. To that end, note that we may identically express \eqref{first_expression} as
\begin{align}
\mbox{minimize} \quad &RE(Z, X) + \frac{1}{2} \epsilon^2 + \frac{1}{2} \eta^2\label{first_expression_prime}\\
\mbox{subject to} \quad &(X \cdot SX)^{-1} \sum_{i=1}^n X(i) (SX)_i (CZ)_i - X \cdot CZ \geq \epsilon \theta^2_X\notag\\
&(CZ)_{i} - X \cdot CZ \leq \epsilon, \quad i = 1, \ldots, n\notag\\
&\ln \left( \sum_{i=1}^n X(i) \exp\{\alpha (CZ)_i\} \right) - \alpha X \cdot CZ \leq \eta^{\theta_X} - 1\notag\\
&\ln \left( \sum_{i=1}^n X(i) \exp\{\alpha (SZ)_i\} \right) - \alpha X \cdot SZ \leq \eta^{\theta_X} - 1\notag\\
&\eta \geq 1, \quad Z \in \mathbb{X}(C),\notag
\end{align}
Let us form the Lagragian of \eqref{first_expression_prime} as
{\allowdisplaybreaks
\begin{align*}
\mathcal{L}(Z, \epsilon, \lambda_0, \lambda, \kappa^1, \kappa^2, \kappa^3, \mu, \nu) &= RE(Z, X) + \frac{1}{2} \epsilon^2 + \frac{1}{2} \eta^2\\
&-\lambda_0 \left( (X \cdot SX)^{-1} \sum_{i=1}^n X(i) (SX)_i (CZ)_i - X \cdot CZ - \epsilon \theta^2_X \right)\\
&+\sum_{i=1}^n \lambda_i \left( (CZ)_{i} - X \cdot CZ - \epsilon \right)\\
&+\kappa^1 \left( \ln \left( \sum_{i=1}^n X(i) \exp\{\alpha (CZ)_i\} \right) - \alpha X \cdot CZ - \eta^{\theta_X} + 1 \right)\\
&+\kappa^2 \left( \ln \left( \sum_{i=1}^n X(i) \exp\{\alpha (SZ)_i\} \right) - \alpha X \cdot SZ - \eta^{\theta_X} + 1 \right)\\
&+\kappa^3 (1-\eta)\\
&-\mu_1 Z(1) - \cdots -\mu_{n} Z(n) +\nu (Z(1) + \cdot + Z(n) - 1).
\end{align*}
}
Since every summand from the second to the last row of the previous Lagrangian expression is convex, we obtain that $\nabla^2_{\hat{Z}\hat{Z}} \mathcal{L}$, where $\hat{Z} = [Z^T, \epsilon, \eta]^T$, is positive definite since, firstly, the relative entropy function is strictly convex (where is it is finite and the relative entropy is everywhere finite as follows from Lemma \ref{feasibility_lemma}) and, secondly, the second derivative of $\frac{1}{2} \epsilon^2$ is equal to one as is the second derivative of $\frac{1}{2} \eta^2$. Therefore, the second order sufficiency condition is satisfied and we can apply \citep{Robinson} (\citep[Theorem 8]{Giorgi}), which implies the lemma. 
\end{proof}

\begin{lemma}
\label{lipschitz_lemma}
$\theta_X > 0$ implies that \eqref{principal_equation} is locally Lipschitz.
\end{lemma}

\begin{proof}
The function $F: \mathbb{X}(C) \times \mathbb{X}(C) \rightarrow \mathbb{R}^n$, where $F_i(X, Z) = X(i) ((CZ)_i - X \cdot CZ), i =1,\ldots, n$, is a polynomial and, therefore, it is Lipschitz. The right-hand-side of \eqref{principal_equation} is the composition $F(X, Z_X)$, which is locally Lipschitz since the composition of locally Lipschitz functions is also locally Lipschitz and, by Lemma \ref{lipschitz_lemma0}, $Z_X$ is locally Lipschitz. This argument completes the proof of our lemma.
\end{proof}

\section{Proof of Theorem \ref{main_theorem}}
\label{proof_of_theorem}

Let us recall a result that is standard in the literature (for example, see \citep{Weibull}), namely, that the Nash equilibria of doubly symmetric bimatrix game, say, $(S, S)$, coincide with the KKT points of the optimization problem
\begin{align*}
\mbox{ maximize } &\quad X \cdot SX\\
\mbox{ subject to } &\quad X \in \mathbb{X}(S).
\end{align*}
Let us also recall that an {\em interior} fixed point of $(S, S)$ is necessarily a Nash equilibrium of $(S, S)$.

\begin{lemma}
\label{pure_gold}
If $C$ is invertible, then $(S, S)$, where $S=CC^T$, has a finite number of fixed points.
\end{lemma}

\begin{proof}
If $C$ is invertible, $S$ is positive definite ($X \cdot SX = X \cdot CC^TX = \|C^TX\|^2$) and, therefore, the quadratic form $X \cdot SX$ is strictly convex. Fixed points of $(S, S)$ that are not pure strategies are necessarily Nash equilibria of their carrier and, therefore, KKT points of their carrier. Let $X^*$ be a fixed point of $(S, S)$ that is not a pure strategy, let $\hat{S}$ be its carrier (i.e., the subgame for that individual face of the simplex), and let $\hat{X}^*$ be the restriction of $X^*$ to $\mathbb{X}(\hat{S})$. If $\hat{X}^*$ is a local maximum of its carrier, then \citep[Proposition 1.4.2]{ConvexAnalysis} implies that $X \cdot \hat{S}X$ must be constant for all $X \in \mathbb{X}(\hat{S})$, which contradicts that $X \cdot SX$ is strictly convex. Therefore, since, by convexity, $\hat{X}^*$ cannot be a saddle point, it must be a global minimum of $X \cdot \hat{S}X$. Since the global minimizer of a strictly convex function is unique, $\hat{X}^*$ must be an isolated fixed point of $(\hat{S}, \hat{S})$. Observe that this implies that $X^*$ must be an isolated fixed point of $(S, S)$ (for if it is not, there must exist some carrier whose fixed points are not isolated, which contradicts that $X \cdot SX$ is strictly convex). Since the fixed points of $(S, S)$ are necessarily isolated points and since a face (carrier) of the simplex cannot carry more than one isolated fixed points (not only in our convex case but also under any payoff matrix), the number of fixed points is also necessarily finite.
\end{proof}

\begin{lemma}
\label{same}
Let $X, X'$ be fixed points of $(C, C^T)$ such that they share the same carrier. If $X$ is a Nash equilibrium, then $X'$ is also a Nash equilibrium.
\end{lemma}

\begin{proof}
If $X, X'$ are fixed points of $(C, C^T)$ that share the same carrier, it is simple to show that $X_\epsilon = (1-\epsilon)X+\epsilon X', \epsilon \in [0, 1]$ is also a fixed point of $(C, C^T)$. But since $(CX_\epsilon)_{\max} - X_\epsilon \cdot CX_\epsilon$ is a continuous function of $\epsilon$, the lemma follows in that either both $X, X'$ are Nash equilibrium strategies or neither of them is.
\end{proof}

In the next lemma, we use an argument close to that used in \citep[p. 119]{Weibull}.

\begin{lemma}
\label{invariance_lemma}
Starting at any $X^0 \in \mathbb{\mathring{X}}(C)$, $\mathbb{X}(C)$ is invariant under \eqref{principal_equation}.
\end{lemma}

\begin{proof}
Let us denote by $X_t(X^0), t \geq 0$ the orbit of \eqref{principal_equation} starting at $X^0 \in \mathbb{\mathring{X}}(C)$. By Lemma \ref{lipschitz_lemma} and the Picard-Lindel\"of theorem, this orbit is locally unique. Note first that, since
\begin{align*}
\sum_{i=1}^n X(i) ((CZ_X)_i - X \cdot CZ_X) = 0,
\end{align*}
the orbit $X_t(X^0), t \geq 0$ lies on the hyperplane
\begin{align*}
\left\{ Y \in \mathbb{R}^n \big| \sum_{i=1}^n Y(i) = 1 \right\}.
\end{align*}
Let us assume for the sake of contradiction that there exists a time $\tau > 0$ such that $X_\tau(X^0)$ is a strategy on the boundary of $\mathbb{X}(C)$. Denote $X^1 \equiv X_\tau(X^0)$. There are now two cases. The first case is that $\theta_{X^1} > 0$. Since $X^1$ is on the boundary, there exists a pure strategy, say $i$, whose probability mass is equal to zero at time $\tau$. But the orbit of $X_t(X_\tau(X^0))$ is also (by Lemmas \ref{fundamental_lemma} (iv) and \ref{lipschitz_lemma} and the Picard-Lindel\"of theorem) unique and it does not fall back in the interior of the simplex because of the multiplicative form of \eqref{principal_equation}. Therefore, the orbit cannot ``touch'' the boundary at a point $X$ where $\theta_X > 0$ in a finite amount of time. The second case is that $\theta_{X^1} = 0$. But then constraint \eqref{third_constraint} in optimization problem \eqref{first_expression} and Lemma \ref{equality_lemma_1} imply that $X$ is a fixed point of \eqref{principal_equation} (for example, $Z_X$ is a pure strategy in the carrier of $X$; constraint \eqref{third_constraint} is always feasible for any $X$ where $\theta_X \geq 0$). Therefore, the orbit stays on $\mathbb{X}(C)$. This completes the proof of the lemma.
\end{proof}

\begin{proof}[Proof of Theorem \ref{main_theorem}]
Let us first note that $X$ is a fixed point of \eqref{principal_equation} if
\begin{align*}
X(i) > 0 \Rightarrow (CZ_X)_i = X \cdot CZ_X.
\end{align*}
We are going to consider two cases, namely, first, one where the initial condition $X^0$ is not a fixed point of \eqref{principal_equation} and, second, one where the initial condition $X^0$ is a fixed point of \eqref{principal_equation}.\\

\noindent
{\bf Case 1:} Note also that
\begin{align*}
\frac{d}{dt} X \cdot SX = \sum_{i=1}^n X(i) (SX)_i (CZ_X)_i - (X \cdot SX) (X \cdot CZ_X).
\end{align*}
Constraint \eqref{first_constraint} then implies that along an orbit of \eqref{principal_equation}, we have that
\begin{align*}
\frac{d}{dt} X \cdot SX \geq 0.
\end{align*}
By the monotone convergence theorem, there exists a positive real number, say $\mathsf{C}$, such that
\begin{align*}
\lim_{t \rightarrow \infty} X \cdot SX = \mathsf{C}.
\end{align*}
Therefore, the function $\mathsf{C} - X \cdot SX$ is a Lyapunov function (nonnegative and monotonically decreasing) for \eqref{principal_equation} and, thus, since, by Lemma \ref{invariance_lemma}, $\mathbb{X}(C)$ is invariant under \eqref{principal_equation}, the LaSalle invariance principle (for example, see \citep[p. 128]{Khalil}) implies that the orbit approaches the largest invariant set contained in the set
\begin{align*}
\left\{ X \in \mathbb{X}(C) \bigg| \sum_{i=1}^n X(i) (SX)_i (CZ_X)_i - (X \cdot SX) (X \cdot CZ_X) = 0 \right\}.
\end{align*}
The omega limit set of the orbit of our equation is a compact and connected set. Let us call this limit set $\mathbb{X}$. Let's first assume that
\begin{align}
\limsup\limits_{t \rightarrow \infty} \theta_{X_t} > 0.\label{proti}
\end{align}
Let $X \in \mathbb{X}$ be such that $\theta_X > 0$. Furthermore, let $\epsilon_X$ denote the optimal value of optimization parameter $\epsilon$ (of optimization problem \eqref{first_expression}) at $X$. By constraint \eqref{first_constraint} we then obtain that $\epsilon_X = 0$ and, therefore, constraint \eqref{third_constraint} implies that $(CZ_X)_{\max} = X \cdot CZ_X$. Lemma \ref{feasibility_lemma} implies then that $RE(Z_X, X)$ is finite, which further implies $X(i) = 0 \Rightarrow Z_X(i) = 0$. But then, since $(CZ_X)_{\max} = X \cdot CZ_X$ implies
\begin{align*}
X(i) > 0 \Rightarrow (CZ_X)_i = (CZ_X)_{\max}, \quad i =1, \ldots, n
\end{align*}
we obtain that
\begin{align*}
Z_X(i) > 0 \Rightarrow X(i) > 0 \Rightarrow (CZ_X)_i = (CZ_X)_{\max}, \quad i = 1, \ldots, n
\end{align*}
and, therefore, that $(CZ_X)_{\max} = Z_X \cdot CZ_X$. That is, $Z_X$ is a symmetric Nash equilibrium strategy. 

If 
\begin{align*}
\liminf\limits_{t \rightarrow \infty} \theta_{X_t} > 0,
\end{align*}
there is nothing else to prove: Every limit point of the multipliers $Z_X, X \in \mathbb{X}$ is a symmetric Nash equilibrium strategy. So let us assume that 
\begin{align*}
\liminf\limits_{t \rightarrow \infty} \theta_{X_t} = 0.
\end{align*}
Then, by constraint \eqref{fourth_constraint} and Lemma \ref{equality_lemma_2}, the limit set $\mathbb{X}$ is a superset of a set, say $\mathbb{X_S}$, such that for each strategy, say $X_S$, in this set, we have that
\begin{align*}
\forall i \in \{1, \ldots, n\} : \left\{ X_S(i) > 0 \Rightarrow (SZ_{X_S})_i = X_S \cdot SZ_{X_S} \right\}.
\end{align*}
Since by Lemma \ref{feasibility_lemma}, $RE(Z_{X_S}, X_S)$ is finite, we obtain that $X_S$ is a fixed point of $(S, S)$ and since, by the assumption that $C$ is invertible and Lemma \ref{pure_gold}, the fixed points of $(S, S)$ are isolated, $\mathbb{X_S}$ is a singleton. Furthermore, by constraint \eqref{third_constraint} and Lemmas \ref{equality_lemma_2} and \ref{feasibility_lemma}, $Z_{X_S}$ is a fixed point of $(C, C^T)$. 

Let us prove that $Z_{X_S}$ is a Nash equilibrium strategy: To that end, note that the limit set $\mathbb{X}$ is connected. Therefore, arbitrarily close to $X_S$ lie strategies whose multipliers are symmetric Nash equilibrium strategies. Furthermore, these strategies have multipliers which have the carrier as $X_S$. But $Z_{X_S}$ also has the same carrier as $X_S$ and, as mentioned earlier, it is a fixed point of $(C, C^T)$. Since Nash equilibrium strategies cannot share the same carrier as non-Nash-equilibrium fixed points (Lemma \ref{same}), our proof $Z_{X_S}$ is also a symmetric Nash equilibrium strategy is complete.

Let's now assume that
\begin{align*}
\limsup\limits_{t \rightarrow \infty} \theta_{X_t} = 0.
\end{align*}
Then, by constraint \eqref{fourth_constraint} and Lemma \ref{equality_lemma_2}, the limit set, say $\mathbb{X_S}$, of \eqref{principal_equation} is such that for each strategy, say $X_S$, in this set, we have
\begin{align*}
\forall i \in \{1, \ldots, n\} : \left\{ X_S(i) > 0 \Rightarrow (SZ_{X_S})_i = X_S \cdot SZ_{X_S} \right\}.
\end{align*}
Since by Lemma \ref{feasibility_lemma}, $RE(Z_{X_S}, X_S)$ is finite, we obtain that $X_S$ is a fixed point of $(S, S)$ and since, by the assumption that $C$ is invertible and Lemma \ref{pure_gold}, the fixed points of $(S, S)$ are isolated, $\mathbb{X_S}$ is a singleton. Furthermore, by constraint \eqref{third_constraint} (or the LaSalle invariance principle), we obtain that $X$ is a fixed point of \eqref{principal_equation}. If $(CZ_X)_{\max} = X \cdot CZ_X$, then, as before, $Z_X$ is a symmetric Nash equilibrium strategy. So let us assume that $(CZ_X)_{\max} > X \cdot CZ_X$. Then, there must exist a pure strategy, say $i$, outside the carrier of $X$ such that
\begin{align}
(CZ_X)_ i > X \cdot CZ_X.\label{asdlkjfasdfjh}
\end{align} 
By the intermediate value theorem, there must exist a neighborhood $O$ of $X$ such that, for all $Y \in O \cap \mathbb{X}(C)$, we have that
\begin{align*}
(CZ_Y)_ i > Y \cdot CZ_Y.
\end{align*}
Note that as follows from a straightforward calculation (for example, see \citep[pp. 98-100]{Weibull}), along $X_t$ we have that
\begin{align*}
\frac{d}{dt} RE(E_i, X_t) = X_t \cdot CZ_{X_t} - E_i \cdot CZ_{X_t}.
\end{align*}
Within the neighborhood $O$ this derivative is strictly negative and the monotone convergence theorem implies that $RE(E_i, X_t)$ converges. But, since $X$ is a fixed point, we obtain that 
\begin{align*}
\lim_{t \rightarrow \infty} \left\{ \frac{d}{dt} RE(E_i, X_t) \right\} = 0
\end{align*}
which implies that
\begin{align*}
(CZ_X)_ i = X \cdot CZ_X
\end{align*}
and contradicts \eqref{asdlkjfasdfjh}. Therefore, it must be that $(CZ_X)_{\max} = X \cdot CZ_X$, and, thus, $Z_X$ is a symmetric Nash equilibrium strategy (by the argument we have demonstrated).\\

\noindent
{\bf Case 2:} If the (interior) initial condition $X^0 \equiv X$ is a fixed point of \eqref{principal_equation}, then, clearly, $(CZ_X)_{\max} - X \cdot CZ_X = 0$ and repeating the previous argument we obtain $(CZ_X)_{\max} - Z_X \cdot CZ_X = 0$. That is, $Z_X$ is a symmetric Nash equilibrium strategy. This completes the proof of our theorem.
\end{proof}

\bibliographystyle{abbrvnat}
\bibliography{real}

\begin{thebibliography}{19}
\providecommand{\natexlab}[1]{#1}
\providecommand{\url}[1]{\texttt{#1}}
\expandafter\ifx\csname urlstyle\endcsname\relax
  \providecommand{\doi}[1]{doi: #1}\else
  \providecommand{\doi}{doi: \begingroup \urlstyle{rm}\Url}\fi

\bibitem[Arora et~al.(2012)Arora, Hazan, and Kale]{AHK}
S.~Arora, E.~Hazan, and S.~Kale.
\newblock The multiplicative weights update method: A meta-algorithm and its
  applications.
\newblock \emph{Theory of Computing}, 8:\penalty0 121--164, 2012.

\bibitem[Avramopoulos(2023)]{AvramopoulosSSRN}
I.~Avramopoulos.
\newblock Computational principles manifesting in learning symmetric equilibria
  by exponentiated dynamics.
\newblock SSRN preprint \url{https://ssrn.com/abstract=4425843}, 2023.

\bibitem[Avramopoulos and Vasiloglou(2023)]{AvramopoulosBoosting}
I.~Avramopoulos and N.~Vasiloglou.
\newblock On algorithmically boosting fixed-point computations.
\newblock arXiv eprint 2304.04665 (cs.GT), 2023.

\bibitem[Bertsekas et~al.(2003)Bertsekas, Nedic, and Ozdaglar]{ConvexAnalysis}
D.~P. Bertsekas, A.~Nedic, and A.~E. Ozdaglar.
\newblock \emph{Convex Analysis and Optimization}.
\newblock Athena Scientific, 2003.

\bibitem[Chen et~al.(2009)Chen, Deng, and Teng]{CDT}
X.~Chen, X.~Deng, and S.~Teng.
\newblock Settling the complexity of computing two-player {N}ash equilibria.
\newblock \emph{Journal of the ACM}, 56\penalty0 (3), 2009.

\bibitem[Daskalakis et~al.(2009)Daskalakis, Goldberg, and
  Papadimitriou]{Daskalakis}
C.~Daskalakis, P.~W. Goldberg, and C.~H. Papadimitriou.
\newblock The complexity of computing a {N}ash equilibrium.
\newblock \emph{SIAM J. Comput.}, 39\penalty0 (1):\penalty0 195--259, 2009.

\bibitem[Giorgi and Zucottti(2018)]{Giorgi}
G.~Giorgi and C.~Zucottti.
\newblock A tutorial on sensitivity and stability in nonlinear programming and
  variational inequalities under differentiability assumptions.
\newblock Dem working paper series \#159 (06-18), Department of Economics and
  Management, University of Pavia, 2018.

\bibitem[Khalil(2002)]{Khalil}
H.~K. Khalil.
\newblock \emph{Nonlinear Systems}.
\newblock Prentice Hall, third edition, 2002.

\bibitem[Lipton et~al.(2003)Lipton, Markakis, and Mehta]{LMM}
R.~Lipton, E.~Markakis, and A.~Mehta.
\newblock Playing large games using simple strategies.
\newblock In \emph{Proc. EC'03}, pages 36--41, 2003.

\bibitem[Milionis et~al.(2023)Milionis, Papadimitriou, Piliouras, and
  Spendlove]{Milionis}
J.~Milionis, C.~Papadimitriou, G.~Piliouras, and K.~Spendlove.
\newblock An impossibility theorem in game dynamics.
\newblock \emph{PNAS}, 120\penalty0 (41), 2023.

\bibitem[Nash(1950)]{Nash}
J.~F. Nash.
\newblock Equilibrium points in $n$-person games.
\newblock \emph{PNAS}, 36\penalty0 (1), Jan. 1950.

\bibitem[Nash(1951)]{Nash2}
J.~F. Nash.
\newblock Non-cooperative games.
\newblock \emph{The Annals of Mathematics, Second Series}, 54\penalty0
  (2):\penalty0 286--295, Sept. 1951.

\bibitem[Papadimitriou(1994)]{PPAD}
C.~H. Papadimitriou.
\newblock On the complexity of the parity argument and other inefficient proofs
  of existence.
\newblock \emph{Journal of Computer and System Sciences}, 48\penalty0
  (3):\penalty0 498--532, 1994.

\bibitem[Piliouras et~al.(2022)Piliouras, Sim, and Skoulakis]{Clairvoyant}
G.~Piliouras, R.~Sim, and S.~Skoulakis.
\newblock Beyond time-average convergence: Near-optimal uncoupled online
  learning via claivoyant multiplicative weights update.
\newblock arXiv eprint 2111.14737 (cs.GT), 2022.

\bibitem[Robinson(1980)]{Robinson}
S.~M. Robinson.
\newblock Strongly regular generalized equations.
\newblock \emph{Mathematics of Operations Research}, 5:\penalty0 43--62, 1980.

\bibitem[Sandholm(2010)]{PopulationGames}
W.~H. Sandholm.
\newblock \emph{Population Games and Evolutionary Dynamics}.
\newblock MIT Press, 2010.

\bibitem[Savani and {von Stengel}(2004)]{Savani-Stengel}
R.~Savani and B.~{von Stengel}.
\newblock Exponentially many steps for finding a {N}ash equilibrium in a
  bimatrix game.
\newblock In \emph{Proc. 45th Annual IEEE Symposium on Foundations of Computer
  Science}, pages 258--267, 2004.

\bibitem[Tsaknakis and Spirakis(2009)]{Tsaknakis-Spirakis-journal}
H.~Tsaknakis and P.~G. Spirakis.
\newblock An optimization approach for approximate {N}ash equilibria.
\newblock \emph{Internet Mathematics}, 5\penalty0 (4):\penalty0 365--382, 2009.

\bibitem[Weibull(1995)]{Weibull}
J.~W. Weibull.
\newblock \emph{Evolutionary Game Theory}.
\newblock MIT Press, 1995.

\end{thebibliography}

\end{document}